\documentclass[12pt]{amsart}
\usepackage{amscd,amssymb,amsthm,amsmath,amssymb,textcomp,supertabular,rotating,longtable,enumerate,rotating,mathrsfs,mathtools,multirow}
\usepackage[matrix,arrow,curve]{xy}
\usepackage{hyperref}

\usepackage{changepage}

\newcommand{\QQ}{{\mathbb Q}}

\newcommand{\PP}{{\mathbb P}}
\newcommand{\CC}{{\mathbb C}}

\newcommand{\Aff}{{\mathbb A}}

\newcommand{\Db}{\mathcal{D}^b}
\renewcommand{\O}{\mathcal{O}}

\newtheorem{theorem}[equation]{Theorem}
\newtheorem*{theorem*}{Theorem}

\newtheorem{proposition}[equation]{Proposition}
\newtheorem*{proposition*}{Proposition}
\newtheorem{lemma}[equation]{Lemma}

\newtheorem*{corollary*}{Corollary}

\newtheorem*{problem*}{Problem}

\newtheorem*{question*}{Question}

\newtheorem*{construction*}{Construction}
\newtheorem*{maintheorem*}{Main Theorem}

\theoremstyle{definition}

\newtheorem*{example*}{Example}

\newtheorem*{definition*}{Definition}

\theoremstyle{remark}
\newtheorem{remark}[equation]{Remark}
\newtheorem*{remark*}{Remark}

\makeatletter\@addtoreset{equation}{section} \makeatother

\author{
Victor Przyjalkowski}

\title[Singularities of Landau--Ginzburg models and derived categories]
{Singularities of Landau--Ginzburg models for complete intersections and derived categories}

\address{\emph{Victor Przyjalkowski}
\newline
\textnormal{Steklov Mathematical Institute of Russian Academy of Sciences, 8 Gubkina street, Moscow 119991, Russia.}
\newline
\textnormal{\texttt{victorprz@mi-ras.ru, victorprz@gmail.com}}}

\begin{document}

\begin{abstract}
Mirror symmetry predicts that bounded derived category of a smooth Fano variety is equivalent to Fukaya--Seidel category of its
Landau--Ginzburg model. It is expected that fibers of Landau--Ginzburg model with ordinary double points correspond
to an exceptional collection of a Fano variety. We verify this expectation on a numerical level for Fano complete intersections
and Calabi--Yau compactifications of their toric Landau--Ginzburg models of Givental's type.
\end{abstract}

\maketitle

\section{Introduction}
\label{section:introduction}
Mirror symmetry relates smooth Fano varieties and Landau--Ginzburg models --- certain families of Calabi--Yau varieties.
Various conjectures relate symplectic geometry of Fano varieties and algebraic one for their Landau--Ginzburg models, and vice versa.
One of the strongest Mirror Symmetry conjectures, Kontsevich's Homological Mirror Symmetry (HMS) conjecture claims the correspondence
in terms of derived categories. More precise, given a Fano variety $X$ and its Landau--Ginzburg model $w\colon Y\to \CC$, one can
equip them by symplectic forms and consider them as algebraic varieties as well as symplectic ones. Algebraic properties
of $X$ are encoded by its bounded derived category of coherent shea\-ves~$\Db(X)$; for $(Y,w)$ one can consider its relative version ---
derived category of singularities $\Db_{\mathrm{Sing}}(Y,w)$. Symplectic properties are encoded by Fukaya category $Fuk(X)$
and Fukaya--Seidel category $FS(Y,w)$, respectively.
HMS %Homological Mirror Symmetry
conjecture claims that these categories are cross-equivalent:
$$
\Db(X)\simeq FS(Y,w),
$$
$$
Fuk(X)\simeq \Db_{\mathrm{Sing}}(Y,w).
$$
We refer to~\cite{Kon94} for details.

Unfortunately, it is very hard to prove %Homological Mirror Symmetry
HMS conjecture even in particular cases.
So to study Mirror Symmetry it is natural to consider another mirror symmetry conjectures and verify some (numerical)
expectations of HMS. Another reason of doing this is that there is no general method to construct a Landau--Ginzburg
model for a given Fano variety. Let us mention Strominger--Yau--Zaslow approach to construct mirror objects
via dual special Lagrangian tori fibrations, see~\cite{SYZ96}, and its algebro-geometric interpretation known as
Gross--Siebert program~\cite{GS19}. However we use another approach: we construct and study Landau--Ginzburg models for Fano varieties
via toric Landau--Ginzburg models, see details in~\cite{Prz18b}. In short, toric Landau--Ginzburg model
for a Fano variety $X$ of dimension~$n$ is a Laurent polynomial $f$ in $n$ variables, satisfying the following.
First, it should combinatorially correspond to
some toric degeneration of the variety $X$. Second, the fa\-mi\-ly~$\{f=\lambda\}$, $\lambda\in \CC$,
has periods that correspond to Gromov--Witten invariants of~$X$. Finally, there should exist a (log) Calabi--Yau compactification,
that is a fiberwise compactification of $\{f=\lambda\}$ to proper family of Calabi--Yau varieties $w\colon Y\to \CC$, such that $Y$ is smooth
and quasi-projective, or even a compactification $u\colon Z\to \PP^1$, where $Z$ is projective, smooth outside $u^{-1}(\infty)$,
and $-K_Z=u^{-1}(\infty)$.

There exist several conjectures related to singular fibers of log Calabi--Yau compactification.
Let us mention a few ones. Let $X$ be a smooth Fano variety of dimension $n$, let~$w\colon Y\to \CC$ be a Calabi--Yau compactification
of its toric Landau--Ginzburg model, and let $u\colon Z\to \PP^1$ be its log Calabi--Yau
compactification. In~\cite{ChP20} it was conjectured that the number
of components of $u^{-1}(\infty)$ equals $\dim |-K_X|$; in the same paper this conjecture was verified for standard Landau--Ginzburg
models of smooth Fano threefolds and of Fano complete intersections.
Another conjecture considers reducible fibers of $w\colon Y\to \CC$.
Let~$k_Y$ be the number of components of reducible fibers minus the number of reducible fibers. The conjecture claims that~$k_Y=h^{1,n-1}(X)$.
It was verified for smooth Fano threefolds (see~\cite{Prz13} and~\cite{Prz17}) and Fano complete intersections
(see~\cite{PSh15}). Note that in the definition of~$k_Y$ multiplicities of components are not taken into account.
However these multiplicities also have an impact on geometry of $X$. That is, in~\cite{KP09} and~\cite{LP25a},~\cite{LP25b}
it was proven that for smooth Fano threefolds a general deformation of the family $(Y,w)$ contains a reducible fiber with non-unipotent monodromy if and only if $X$ is rational. In almost all cases all reducible fibers are either non-reduced or reduced and have simple normal crossings.
It can be derived from Griffiths--Landman--Grothendieck theorem (see~\cite{Ka70}) and Mumford's semistable reduction theorem
(see~\cite{KKMSD73}) that the monodromy in the first case is not unipotent, and in the second case it is unipotent.
(In some cases one needs to compute the monodromy directly.)
Finally, the reducible fibers conjecture has a deep generalization called Katzarkov--Kontsevich--Pantev conjecture,
see~\cite{KKP17}. It is motivated by HMS conjecture
and can be considered as a ``linearization'' of its part.
The conjecture claims equality of Hodge numbers $h^{p,q}(X)$ and $f^{n-p,q}(Y,w)$, where $f^{n-p,q}(Y,w)$
are Hodge-type numbers for $(Y,w)$. For del Pezzo surfaces this conjecture were proved in~\cite{LP16},
and for Fano threefolds in~\cite{ChP25}.

Thus, we considered the fiber over infinity and reducible fibers.
Since objects of~$FS(Y,w)$ are vanishing (to singularities of fibers) Lagrangians,
the rest ``interesting'' fibers of  $(Y,w)$ are singular,
but irreducible. Often this means that such fibers contain one ordinary double point.
The simplest case is a Lefschetz family, that is one whose fibers are smooth except for
$k$ ones which have $1$ ordinary double point in each one.
In~\cite{Se08} a subca\-te\-go\-ry~$\mathcal{C}$ of $FS(Y,w)$ was constructed and it was proved that it has
a semiorthogonal decomposition into $k$ objects:
$$
\mathcal{C}=\langle \mathcal{L}_1,\ldots,  \mathcal{L}_k\rangle.
$$
According to\cite{GPS20, GPS24a, GPS24b},  %private communication with M.\,Abouzaid,
the category $\mathcal{C}$ is ``large enough''.
This imp\-lies~$\mathcal{C}\simeq FS(Y,w)$.
Thus, HMS conjecture predicts that $\Db(X)$ has a full exceptional collection consisting
of $k$ objects.

Unfortunately, a Landau--Ginzburg model of a Fano variety rarely is of Lefschetz type.
However if it has $k$ fibers having only one ordinary double point, then in~\cite{AAK16}
it was constructed a subca\-te\-go\-ry~$\mathcal{C}$ of $FS(Y,w)$
such that
$$
\mathcal{C}=\langle \mathcal{R}_Y, \mathcal{L}_1,\ldots,  \mathcal{L}_k\rangle,
$$
where $\mathcal{R}_Y$ is some category, and $\mathcal{L}_i$ are some objects. It is expected
that $\mathcal{C}$ is ``large enough'' which implies
$
\mathcal{C}=FS(Y,w)
$.
(Proving this expectation has strong symplectic difficulties though.)
In particular, HMS conjecture implies that $\Db(X)$ has a semiorthogonal decomposition
with $k$ exceptional objects as well.
In this paper we prove the evidence for this expectation for Fano complete intersections.

More precise, let $X\subset \PP^N$ be a smooth Fano complete intersection of hypersurfaces
of degrees $d_1,\ldots, d_k$. Let $n=\dim(X)$ and let $i_X=N+1-\sum d_i$ be its Fano index.
The following is well known.
\begin{proposition}
\label{proposition:Fano semiorthogonal decomposition}
One has
$$
\Db(X)\simeq \langle\mathcal{R}_X,\O_X,\ldots,\O_X(i_X-1)\rangle.
$$
\end{proposition}

A \emph{toric Landau--Ginzburg model of Givental's type} for $X$
is a Laurent polynomial
$$
f_{X}=\frac{\prod_{i=1}^k(x_{i,1}+\ldots+x_{i,d_i-1}+1)^{d_i}}{\prod_{i=1}^k \prod_{j=1}^{d_i-1} x_{i,j}\prod_{j=1}^{i_X-1} y_j}+y_1+\ldots+y_{i_X-1}\in \CC[x_{i,j}^{\pm 1}, y_{s}^{\pm 1}].
$$
In~\cite{PSh15} it was proved that $f_X$ admits a Calabi--Yau compactification, and that its \emph{central fiber} (i.\,e. the fiber over $0$)
consists of $h^{1,n-1}(X)$ components. The main result of the paper is the following.
\begin{theorem}
\label{theorem:LG ODP}
Let $X$ be a smooth Fano complete intersection.
Let $(Y,w)$ be a Calabi--Yau compactification of its toric Landau--Ginzburg
model of Givental's type. Then, in addition to possibly singular central fiber, there are exactly $i_X$ singular fibers,
and all of them have a unique ordinary double point.
\end{theorem}

\begin{remark}
\label{remark:critical values}
The proof of~\ref{theorem:LG ODP} implies that if $X$ is a complete intersection of hypersurfaces of degrees $d_1,\ldots,d_k$,
and $d=d_1^{d_1}\cdot\ldots\cdot d_k^{d_k}$,
then coordinates of
the singular points are~$\varepsilon i_Xd^{\frac{1}{i_X}}$, where $\varepsilon$ is a primitive $i_X$-th root of unity.
\end{remark}

\begin{remark}
Recall that for a smooth Fano weighted complete intersection $X$ an analog of toric Landau--Ginzburg model $f_X$ was
constructed in~\cite{Prz07}.
One can easily check that periods for $f_X$ coincide with regularized generating series of Gromov--Witten invariants for $X$.
Moreover, by~\cite[Theorem 2.2]{ILP13}, the polynomial $f_X$ corresponds to a toric degeneration of $X$. However it is not clear how to extend
proofs of existence of Calabi--Yau compactifications from~\cite{PSh15} or~\cite{Prz18a} to the weighted case.
Moreover, if $X$ is quasi-smooth, to extend the definition of toric Landau--Ginzburg model to $X$ one needs to use
orbifold Gromov--Witten invariants instead of the usual ones.
Conjecturally, generating
series $\widetilde{I}^X_0$ for one-pointed orbifold Gromov--Witten invariants for a quasi-smooth weighted complete intersection looks similar to
one for a smooth weighted complete intersection.
Following~\cite{Prz07}, one can construct a suggestion $f_X$ for toric Landau--Ginzburg model for a quasi-smooth weighted complete
intersection. Its periods coincide with $\widetilde{I}^X_0$,
and, again by~\cite[Theorem 2.2]{ILP13}, the polynomial $f_X$ corresponds to a toric degeneration of~$X$.
Thus one can consider $f_X$ as a natural suggestion for toric Landau--Ginzburg model for $X$.

In his thesis C.\,Avila considered $\QQ$-Fano index $1$ threefold weighted hy\-per\-sur\-fa\-ces~$X$,
that is, $\QQ$-factorial terminal hypersurfaces $X$ of degree $a=a_1+a_2+a_3+a_4$ in~$\PP(1,a_1,a_2,a_3,a_4)$, see~\cite{Av24}.
There are $95$ families of such hypersurfaces~\cite{IF00}.
C.\,Avila used an alternative approach to construct a Calabi--Yau compactification of $f_X$ and proved analog of Theorem~\ref{theorem:LG ODP} for it. More precise, the family given by $f_X$ has $a$-to-$1$ cover of Landau--Ginzburg model $f_\PP$ of the weighted projective space
$\PP=\PP(a_1,a_2,a_3,a_4)$, cf.~\cite[Remark 4.5]{ILP13}. One can prove that $f_\PP$ has a Calabi--Yau compactification
(cf.~\cite[Theorem 1.15]{CG11}); thus, $f_X$ has a Calabi--Yau compactification (possibly except for the fiber over $0$) as well.
One can also prove that the fiber of the compactification over $0$ also lies in the anticanonical linear system. Finally,
since there are exactly $a$ non-central singular fibers of the compactification of $f_\PP$, and all of them has a single node,
the compactification of $f_X$ has $1=i_X$ singular fiber away from $0$, and this fiber has one node.
\end{remark}

\medskip

{\bf Acknowledgements.}
This work was supported by the Russian Science Foundation under grant no. 25-11-00057, \url{https://rscf.ru/en/project/25-11-00057/}.

\section{Proof of the main theorem}
\label{section:main theorem}
Note that different Calabi--Yau compactifications of a given Laurent polynomials differ by flops (see~\cite[Theorem 1]{Ka08}), and these flops preserves ordinary double points
in fibers. Thus, to prove Theorem~\ref{theorem:LG ODP}, we may choose any compactification we would like.

Let $X\subset\PP^N$ be a Fano complete intersection of hypersurfaces of degrees $d_1,\ldots,d_k$, let $\dim(X)=N-k=n$, and let $i_X=N+1-\sum d_i$.
Put
$$
f_{X}=\frac{\prod_{i=1}^k(x_{i,1}+\ldots+x_{i,d_i-1}+1)^{d_i}}{\prod_{i=1}^k \prod_{j=1}^{d_i-1} x_{i,j}\prod_{j=1}^{i_X-1} y_j}+y_1+\ldots+y_{i_X-1}\in \CC[x_{i,j}^{\pm 1}, y_{s}^{\pm 1}].
$$
We follow compactification procedure described in~\cite{PSh15}. That is, for $i_X>1$ we consider 
\begin{equation*}\label{eq:global-resolution-Y-prime}
F=y_0^{i_X}\prod_{i=1}^k(x_{i,1}+\ldots+x_{i,d_i})^{d_i}-(\lambda y_0-y_1-\ldots-y_{i_X-1}){\prod_{i=1}^k \prod_{j=1}^{d_i} x_{i,j}\prod_{s=1}^{i_X-1} y_s}
\end{equation*}
in
$$
\PP^{d_1-1}\times\ldots,\PP^{d_k-1}\times \PP^{i_X-1}\times \Aff^1
$$
with coordinates $x_{i,j}$, $y_s$, $\lambda$; here $i=1,\ldots,k$, $j=1,\ldots,d_i$, $s=0,\ldots,i_X-1$.
For $i_X=1$ the hypersurface $Y'$ is given by
\begin{equation*}\label{eq:global-resolution-Y-prime}
F=\prod_{i=1}^k(x_{i,1}+\ldots+x_{i,d_i})^{d_i}-\lambda \prod_{i=1}^k \prod_{j=1}^{d_i} x_{i,j}
\end{equation*}
in
$$
\PP^{d_1-1}\times\ldots,\PP^{d_k-1}\times \Aff^1
$$
with coordinates $x_{i,j}$, $\lambda$; here $i=1,\ldots,k$, $j=1,\ldots,d_i$, $s=0,\ldots,i_X-1$.
Obviously, $Y'$ with the natural projection $w'\colon Y'\to \Aff^1$ is a compactification of the family $\{f_X=\lambda\}$.
To get a Calabi--Yau compactification of $f_X$, we need to crepantly resolve singularities of~$Y'$.

\begin{lemma}
\label{lemma: only iX points}
Let $w\colon Y\to \Aff^1$ be a crepant resolution of $w'\colon Y'\to \Aff^1$. Then
$w^{-1}(\lambda)$ for~$\lambda\neq 0$ is smooth unless $\lambda^{i_X}=i_X^{i_X}\prod d_i^{d_i}$;
in the latter case
$w^{-1}(\lambda)$ has exactly one singular point.
In particular, $w\colon Y\to \Aff^1$ has $i_X$ non-central singular fibers, and these fibers have a unique singular point.
\end{lemma}

\begin{proof}
Vanishing the derivative of $F$ by $\lambda$, which is equal to
$\prod_{j=1}^{d_i} x_{i,j}\prod_{s=0}^{i_X-1} y_s$, we see that the exceptional locus of a resolution of
singularities of $Y'$ lies over the locus
$$\left\{\prod_{j=1}^{d_i} x_{i,j}\prod_{s=0}^{i_X-1} y_s=0\right\}$$
if $i_X>1$, and $\{\prod_{j=1}^{d_i} x_{i,j}=0\}$ if $i_X=1$.
In particular, in the complement of this locus, that is, when all coordinates $x_{i,j}$ and $y_s$
are non-zero, the resolution is an isomorphism. Let us first check that the fibers of the resolution
are smooth in the neighborhood of the exceptional divisor, possibly except for the central fiber.
For this we follow the proof of~\cite[Theorem 5.1]{PSh15}, providing more details.

If $i_X>1$, let us consider the local chart $y_0=1$. Without loss of generality in the neighborhood of singular locus
one may assume that
$$x_{1,1}+\ldots+x_{1,d_1}=0,\ \  x_{1,d_1}=\ldots=x_{k,d_k}=1,\ \  \mathrm{and}\ \  x_{1,d_1-1}\neq 0.
$$
Renumbering variables if needed, one may also assume that there exist the number $r$ such that
$x_{i,d_i-1}\neq 0$ for $i\leq r$ and
$x_{i,1}=\ldots=x_{i,d_i-1}=0$ for $i> r$. Note that either we have~$y_1=\ldots=y_{i_X-1}=0$,
or for some $j$ we have $y_j\neq 0$; without loss of generality we may assume that $j={i_X-1}$.
In the latter case consider the set
$$
\{x_{i,j}\mid i=1,\ldots, r,\ \ j=1,\ldots,d_i-2\}\cup \{x_{i,j}\mid i=r+1,\ldots, k,\ \  j=1,\ldots,d_i-1\}.
$$
The cardinality of this set is $n-r-{i_X-1}$. Denote its elements by $x_{i_X},\ldots, x_{n-r}$.
Consider the analytic change of variables
\begin{multline*}
u_1=\left(x_{1,1}+\ldots+x_{1,d_1-1}+1\right)\cdot \prod_{i=r+1}^{k}(x_{i,1}+\ldots+x_{i,d_i-1}+1)^{\frac{d_i}{d_1}}
\cdot \prod_{i=1}^{r} x_{i,d_i-1}^{-\frac{1}{d_1}}\cdot y_l^{-\frac{1}{d_1}},\\
u_i=x_{i,1}+\ldots+x_{i,d_i-1}+1 \ \ \mbox{for } i=2,\ldots, r, \\
v_j=y_j\ \ \mbox{for } j=1,\ldots, i_X-2, \ \  v_{i_X-1}=\lambda-y_1-\ldots-y_{i_X-1},\ \
v_{s}=x_s \ \ \mbox{for } s=i_X,\ldots, n-r.
\end{multline*}
After this change of variables $Y'$ in our neighborhood is given by
$$
\prod_{i=1}^{r}u_i^{d_i}=\prod_{j=1}^{n-r} v_j.
$$
In particular, it has a crepant resolution, see~\cite[Resolution procedure 4.4]{PSh15}.
Moreover, these resolutions in different neighborhoods agree with each other. Finally, the equations of the hypersurface in these charts
do not depend on $\lambda$. This shows that in the neighborhood of exceptional divisor of the resolution
fibers of $w\colon Y\to \Aff^1$ are smooth.

Now consider 
the neighborhood of the locus
$$
y_0\neq 0, \quad
y_1=\ldots=y_{i_X-1}=0
$$
(
we allow here the case $i_X=1$ omitting the variables $y_j$ in this case).
The equation of $Y'$ in this neighborhood
can be locally analytically rewritten as the equation
\begin{equation*}\label{eq:LG-hypersurface-simplified}
\prod_{i=1}^k(x_{i,1}+\ldots+x_{i,d_i})^{d_i}=\lambda y_1\cdot\ldots\cdot
y_{i_X-1}\cdot \prod\limits_{i=1}^k \prod_{j=1}^{d_i}x_{i,j}.
\end{equation*}
The variety defined by this equation again can be crepantly resolved, see the proof of~\cite[Theorem 5.1]{PSh15},
and in the neighborhood of the exceptional divisor fibers of~$w\colon Y\to \Aff^1$ are smooth, except possibly for the central one.

Finally, for the case $i_X>1$ consider the neighborhood of $y_0=0$. We may assume that~$y_{i_X-1}=1$.
Obviously, in the neighborhood of any point of $Y'$ in this locus
one either has~\mbox{$y_j\neq 0$} for some index~\mbox{$j\in \{1,\ldots,i_X-2\}$},
or
$$
\lambda y_0-y_1-\ldots-y_{i_X-2}-1\neq 0.
$$
In the former case consider the change of coordinates
\begin{multline*}
u_0=y_0{y_j}^{-\frac{1}{i_X}},\quad v_1=y_1, \quad \ldots, \quad v_{j-1}=y_{j-1},\quad v_j=\lambda y_0-y_1- %\ldots -y_{j-1}-y_{j+1}
\ldots-y_{i_X-2}-1, \quad \\ v_{j+1}=y_{j+1},\quad \ldots, \quad v_{i_X-2}=y_{i_X-2}.
\end{multline*}
In the latter case consider the change of coordinates
$$
u_0=y_0\left(\lambda y_0-y_1-\ldots-y_{i_X-2}-1\right)^{-\frac{1}{i_X}}, \quad v_1=y_1, \quad \ldots, \quad v_{i_X-2}=y_{i_X-2}.
$$
Thus in the new coordinates the hypersurface $Y'$ is defined by equation
\begin{equation*}
u_0^{i_X}\prod_{i=1}^k(x_{i,1}+\ldots+x_{i,d_i})^{d_i}={\prod_{i=1}^k \prod_{j=1}^{d_i} x_{i,j}\prod_{j=1}^{i_X-2} v_j}.
\end{equation*}
Now, proceeding with change of variables $x_{i,j}$ as above, we again arrive to equations to
which~\cite[Resolution procedure 4.4]{PSh15} can be applied.
Again, in the neighborhood of the exceptional divisor fibers of $w\colon Y\to \Aff^1$ are smooth.

Summarizing, the resolution of $Y'$ described above shows that in the neighborhood of exceptional divisor of the resolution
all fibers, possibly except for the central one, are smooth, and the resolution is an isomorphism
in the complement to the set
$$\left\{\prod_{j=1}^{d_i} x_{i,j}\prod_{s=0}^{i_X-1} y_s=0\right\}$$
if $i_X>1$, and $\{\prod_{j=1}^{d_i} x_{i,j}=0\}$ if $i_X=1$.
Now let us find fibers which are singular in this complement.
For $i_X=1$ vanishing of $\frac{\partial F}{\partial x_{i,j}}$ implies that for all possible $i,j,l$ one has
$x_{i,j}=x_{i,l}$. Putting these equalities to $F$ one gets $\prod_{i=1}^k x_{i,1}^{d_i}-\lambda=0$,
which implies that a unique singular fiber, except for the central one, is the fiber over $\prod_{i=1}^k d_i^{d_i}$,
and it contains exactly one singular point. For $i_X>0$ in the neighborhood of singularities
one also have~$\lambda y_0-y_1-\ldots-y_{i_X-1}\neq 0$, because otherwise
the derivative $\frac{\partial F}{\partial y_1}$ equals
$$
-\prod_{j=1}^{d_i} x_{i,j}\prod_{s=1}^{i_X-1} y_s=0,
$$
which is non-zero.
Vanishing partial derivatives of $F$ by $x_{i,j}$ and $x_{i,l}$, one sees that at a singular point one should
have~$x_{i,j}=x_{i,l}$ for all possible $i$, $j$, $l$. Similarly, one should have~$y_r=y_s$ for $s,r=1,\ldots,i_X-1$.

Denote $\prod_{i=1}^k d_i^{d_i}$ by $d$.
Let $P$ be a singular point of a singular fiber $w^{-1}(\lambda)$.
We may assume that its $x_{i,j}$-coordinate equals $1$, $y_0$-coordinate equals $1$,
and $y_s$-coordinate equals~$y\in \CC$ for $s>0$.
Note that the equation $F=0$ at $P$ is
$$
d=(\lambda-i_X+1)y^{i_X-1}.
$$
Vanishing $\frac{\partial F}{\partial x_{i,j}}$ gives the same equation
$$
d=(\lambda-i_X+1)y^{i_X-1}.
$$
Vanishing $\frac{\partial F}{\partial y_s}$ for $s>0$ gives the equation
$$
(\lambda-(i_X-1)y)y^{i_X-2}-y^{i_X-1}=0,
$$
so that $\lambda=i_Xy$.
Finally, vanishing $\frac{\partial F}{\partial y_0}$ gives the equation
$$
di_X=\lambda y^{i_X-1}.
$$
These equations have $i_X$ solutions: if $\varepsilon$ is a primitive $i_X$-th root of unity,
then $y=\varepsilon^r{d}^{\frac{1}{i_X}}$ and $\lambda=\varepsilon^ri_X{d}^{\frac{1}{i_X}}$. This completes the proof of the lemma.
\end{proof}

To complete the proof of Theorem~\ref{theorem:LG ODP}, we need to check that the singular points are ordinary double ones.
For this we will use the following easy auxiliary lemma.

\begin{lemma}
\label{lemma:easy lemma}
Consider an $(n\times n)$-matrix
$$
M\left(
  \begin{array}{ccccc}
    b & a & \ldots & \ldots & a \\
    a & b & \ldots & \ldots & a \\
    \ldots & \ldots & \ldots & \ldots & \ldots \\
    \ldots & \ldots & \ldots & b & a \\
    a & a & \ldots & a & b \\
  \end{array}
\right).
$$
Then $\det M=0$ if and only if either $a=b$, or $(n-1)a+b=0$.
\end{lemma}

\begin{proof}
Let $l_1,\ldots l_n$ be the rows of $M$. Then $M$ is degenerate if and only if
there exists coefficients $\alpha_1,\ldots,\alpha_n$ such that they do not vanish altogether and
$$
\alpha=\alpha_1l_1+\ldots+\alpha_n l_n=0.
$$
If $a=b$, or $(n-1)a+b=0$, then one can take $\alpha_1=\ldots=\alpha_{n-1}=1$, $\alpha_n=1-n$
and~$\alpha_1=\ldots=\alpha_{n}=1$, respectively.
On the other hand, summing up coordinates of $\alpha$ one gets $(\sum \alpha_i)\cdot((n-1)a+b)$. Thus,
either $(n-1)a+b=0$, or $\sum \alpha_i=0$. In the latter case without loss of generality we may assume $\alpha_1\neq 0$.
Then
$$
0=\alpha_1b+\alpha_2a+\ldots+\alpha_na-(\sum \alpha_i)a=(b-a)\alpha_1,
$$
which implies $a=b$.
\end{proof}

Now let us check the type of the singularity of non-central singular fibers.

\begin{lemma}
\label{lemma: ODP}
Let $P$ be a singular point of a non-central fiber of $w\colon Y\to \Aff^1$.
Then $P$ is an ordinary double point.
\end{lemma}

\begin{proof}
Let $P\in w^{-1}(\lambda)$.
Denote $d=\prod d_i^{d_i}$ and $\alpha={d}^{\frac{1}{i_X}}$, where the root is chosen such that $\lambda=i_X\alpha$.
Proof of Lemma~\ref{lemma: only iX points} implies that we may assume that $x_{i,j}$-coordinates and~$y_0$-coordinates of $P$ equals $1$
and $y_s$-coordinates for $s>0$ equal $\alpha$. Let us make a change of coordinates sending $P$ to the origin,
that is, put
$$
x_{i,j}=a_{i,j}+1\ \mbox{and}\ x_{i,d_i}=1\ \mbox{for}\ i\in 1,\ldots, k,\ j=1,\ldots, d_i, \ \  \ \ \ y_0=1,\ \ \ y_s=b_s+\alpha \ \mbox{for}\ s>0.
$$
Then the equation of $Y'$ in the new coordinates is given by vanishing of
\begin{multline*}
G=\prod_{i=1}^k(a_{i,1}+\ldots+a_{i,d_i-1}+d_i)^{d_i}%\cdot (b_0+1)^{i_X}-
%\\
+(%i_X\alpha b_0
b_1+\ldots+b_{i_X-1} - \alpha)\cdot
\prod_{i=1}^k \prod_{j=1}^{d_i-1} (a_{i,j}+1)\prod_{s=1}^{i_X-1} (b_s+\alpha).
\end{multline*}

Since $P$ is singular, constant and linear terms of $G$ vanish. To prove that $P$ is an ordinary double point
we need to check that quadratic term of $G$ is non-degenerate.

Let us find the quadratic term. 
The coefficient at $b_sb_r$ for all possible $s,r$ equals $\alpha^{i_X-2}$.
The coefficient at $a_{i,j}b_s$ equals $0$ for all possible $s$.
For $i\neq l$ the coefficient at $a_{i,j}a_{l,m}$ equals $0$.
Finally, put $D_i=\frac{(d_i-1)d}{d_i}$. %; note that it is integral because $d_i>1$.
Then the coefficient at $a_{i,j}a_{i,l}$ equals $D_i-d$ for all possible $i$ and $j\neq l$, and the coefficient at
$a_{i,j}a_{i,j}$ equals $D_i$. Denote $M_i=D_i-d$.
Thus, in coordinates $a_{1,1},\ldots, a_{k,d_k-1},b_1,\ldots,b_{i_X-1}$ the matrix of the quadratic form
given by the quadratic term is
{%\footnotesize
$$
\left(
  \begin{array}{cccc|ccc|cccc|ccc}
    2D_1 & M_1 & \ldots & M_1 & 0 & \ldots & 0 & 0 & \ldots & \ldots & 0 & 0 & \ldots & 0 \\
    M_1 & 2D_1 & \ldots & \ldots & 0 & \ldots & 0 & 0 & \ldots & \ldots & 0 & 0 & \ldots & 0 \\
    \ldots & \ldots & \ddots & M_1 & 0 & \ldots & 0 & 0 & \ldots & \ldots & 0 & 0 & \ldots & 0 \\
    M_1 & \ldots & M_1 & 2D_1 & 0 & \ldots & 0 & 0 & \ldots & \ldots & 0 & 0 & \ldots & 0 \\
\hline
    0 & \ldots & \ldots & 0 & \ldots & \ldots & \ldots & 0 & \ldots & \ldots & 0 & 0 & \ldots & 0 \\
    0 & \ldots & \ldots & 0 & \ldots & \ldots & \ldots & 0 & \ldots & \ldots & 0 & 0 & \ldots & 0 \\
    0 & \ldots & \ldots & 0 & \ldots & \ldots & \ldots & 0 & \ldots & \ldots & 0 & 0 & \ldots & 0 \\
\hline
    0 & \ldots & \ldots & 0 & 0 & \ldots & 0 & 2D_k & M_k & \ldots & M_k & 0 & \ldots & 0 \\
    0 & \ldots & \ldots & 0 & 0 & \ldots & 0 & M_k & 2D_k & \ldots & \ldots & 0 & \ldots & 0 \\
    0 & \ldots & \ldots & 0 & 0 & \ldots & 0 & \ldots & \ldots & \ddots & M_k & 0 & \ldots & 0 \\
    0 & \ldots & \ldots & 0 & 0 & \ldots & 0 & M_k & \ldots & M_k & 2D_k & 0 & \ldots & 0 \\
\hline
    0 & \ldots & \ldots & 0 & \ldots & \ldots & \ldots & 0 & \ldots & \ldots & 0  & 2\alpha^{i_X-2} & \ldots & \alpha^{i_X-2} \\
    0 & \ldots & \ldots & 0 & \ldots & \ldots & \ldots & 0 & \ldots & \ldots & 0 & \alpha^{i_X-2} & \ddots & \alpha^{i_X-2} \\
    0 & \ldots & \ldots & 0 & \ldots & \ldots & \ldots & 0 & \ldots & \ldots & 0 & \alpha^{i_X-2} & \ldots & 2\alpha^{i_X-2} \\
  \end{array}
\right).
$$
}
By Lemma~\ref{lemma:easy lemma}, this matrix is non-degenerated.
\end{proof}

Thus, lemmas~\ref{lemma: only iX points} and~\ref{lemma: ODP} imply Theorem~\ref{theorem:LG ODP}.

\begin{remark}
Let $X$ be a Fano complete intersection of dimension $n>2$. Then $\Db(X)$ has a full exceptional collection if and only if $X$
is either a quadric, or $n$ is even and $X$ is a complete intersection of two quadrics.
The lengths of exceptional collections are $n+1$ if $X$ is odd-dimensional quadric, $n+2$ if $X$ is an even-dimensional
quadric, and $2n+4$ if $X$ is an even-dimensional intersection of two quadrics, see~\cite{Ka88}.
Thus, we expect that the central fiber
of Landau--Ginzburg model for $X$ has non-isolated singularities unless~$X$ is of one of the types above;
in these cases the central fiber should contain $1$, $2$, or $2n+5$ ordinary double points, respectively.
Using compactification procedure described in~\cite{PSh15}, X.\,Hu verified this expectation for quadrics~\cite{Hu22}.
\end{remark}

\begin{remark}
Let $X$ be a Fano complete intersection of hypersurfaces of degrees $d_1,\ldots,d_k$, let $\dim X=n$, and let $i_X=n+k+1-\sum d_i$ be the index of $X$.
Let $d=\prod d_i^{d_i}$. According to~\cite{Be95}, quantum cohomology $QH^*(X)$ is generated by
a class of hyperplane sec\-ti\-on~$h$ and %as $QH*(X)=\CC[h]/h^n-d q^n\oplus H_o(X)$,
%where the
a primitive cohomology $H_o(X)$ (that is, $H^n(X)$ if $n$ is odd and orthogonal to~$h^{\frac{n}{2}}$
in~$H^{\frac{n}{2}}(X)$ if $n$ is even). The relations defining the small quantum cohomology ring~$QH^*(X)$ are
$$
h^{\star(n+1)}=dqh^{n+1-i_X}\ \ \mbox{and}\ \  \alpha \star \beta=(\alpha\cdot \beta) \cdot \frac{1}{d_1\cdot\ldots\cdot d_k}(h^n-d q h^{n-i_X}),
$$
where $q$ is a quantum parameter and $\alpha,\beta\in H_o(X)$. Thus, Theorem~\ref{theorem:LG ODP} and Remark~\ref{remark:critical values} confirm the expectation that the Milnor ring
of Landau--Ginzburg model of Fano variety is isomorphic to its small quantum cohomology ring. (For the case of quadrics this
is proved in~\cite[Corollary 3.2]{Hu22}.)

\end{remark}


\begin{thebibliography}{999}

\bibitem[AAK16]{AAK16}
M.\,Abouzaid, D.\,Auroux, L.\,Katzarkov,
\emph{Lagrangian fibrations on blowups of toric varieties and mirror symmetry for hypersurfaces},
Publ. Math. Inst. Hautes \'Etudes Sci., \textbf{123} (2016), 199--282.

\bibitem[Av24]{Av24}
C.\,Avila,
\emph{Mirror Symmetry for $\QQ$-Fano threefolds}, PhD thesis.

\bibitem[Be95]{Be95}
A.\,Beauville,
\emph{Quantum cohomology of complete intersections},
Mat. Fiz. Anal. Geom., \textbf{2}:3-4 (1995), 384--398.

\bibitem[ChP22]{ChP20}
I.\,Cheltsov, V.\,Przyjalkowski, {\it Fibers over infinity of Landau--Ginzburg models}, Commun. Number Theory Phys., \textbf{16}:4 (2022), 673--693.

\bibitem[ChP25]{ChP25}
I.\,Cheltsov, V.\,Przyjalkowski, {\it Katzarkov--Kontsevich--Pantev Conjecture for Fano threefolds}, %arXiv:1809.09218.
Proc. Steklov Inst. Math., \textbf{328} (2025), 1--156.

\bibitem[CG11]{CG11}
A.\,Corti, V.\,Golyshev, \emph{Hypergeometric equations and weighted projective spaces}, Sci. China Math., \textbf{54} (2011), 1577--1590

\bibitem[GPS20]{GPS20}
S.\,Ganatra, J.\,Pardon, V.\,Shende, \emph{Covariantly functorial wrapped Floer
theory on Liouville sectors}, Publ. Math. Inst. Hautes \'Etudes Sci., \textbf{131} (2020), 73--200.

\bibitem[GPS24a]{GPS24a}
S.\,Ganatra, J.\,Pardon, V.\,Shende, \emph{Microlocal Morse theory of wrapped Fukaya categories}, Ann. of Math. (2), \textbf{199}:3
(2024), 943--1042.

\bibitem[GPS24b]{GPS24b}
S.\,Ganatra, J.\,Pardon, V.\,Shende, \emph{Sectorial descent for wrapped Fukaya categories}, J. Amer. Math. Soc., \textbf{37}:2
(2024), 499--635.


\bibitem[GS19]{GS19}
M.\,Gross, B.\,Siebert,
{\em Intrinsic mirror symmetry}, {arXiv:1909.07649. 2019.}

\bibitem[Hu22]{Hu22}
X.\,Hu, \emph{Mirror symmetry for quadric hypersurfaces}, arXiv:2204.07858.

\bibitem[Ian00]{IF00}
A.\,R.\,Iano-Fletcher, \emph{Working with weighted complete intersections},
Explicit birational geometry of 3-folds, London Math. Soc. Lecture Note Ser., \textbf{281}, 101--173, Cambridge Univ. Press, Cambridge, 2000.

\bibitem[ILP13]{ILP13}
N.\,Ilten, J.\,Lewis, V.\,Przyjalkowski, \emph{Toric degenerations of Fano threefolds giving weak Landau--Ginzburg models},
J. Algebra, \textbf{374} (2013), 104--121.


\bibitem[Ka88]{Ka88}
M.\,Kapranov, \emph{On the derived categories of coherent sheaves on some homogeneous spaces},
Invent. Math., \textbf{92}:2 (1988), 479--508.

\bibitem[Ka70]{Ka70}
N.\,Katz, {\it Nilpotent connections and the monodromy theorem: applications of a result of Turrittin},
IH\'ES Publ. Math., \textbf{39} (1970), 175--232.

\bibitem[KKP17]{KKP17}
L.\,Katzarkov, M.\,Kontsevich, T.\,Pantev,
{\it Bogomolov--Tian--Todorov theorems for Landau--Ginzburg models}, J. Diff. Geom., \textbf{105}:1 (2017), 55--117.

\bibitem[KP09]{KP09}
L.\,Katzarkov, V.\,Przyjalkowski,
{\it Generalized Homological Mirror Symmetry and cubics}, Proc. Steklov Inst. Math., \textbf{264}:6 (2009), 87--95.

\bibitem[Ka08]{Ka08}
Y.\,Kawamata, \emph{Flops connect minimal models}, Publ. Res. Inst. Math. Sci., \textbf{44}:2 (2008),
419--423.

\bibitem[KKMSD73]{KKMSD73}
G.\,Kempf, F.\,Knudsen, D.\,Mumford, B.\,Saint-Donat,
\emph{Toroidal embeddings I}, Lect. Notes Math., 1973.

\bibitem[Kon94]{Kon94}
M.\,L.\,Kontsevich, {\it Homological algebra of mirror symmetry}, Proc.
International Congress of Mathematicians (Z\"{u}rich 1994),
Birkh\"{a}uzer, Basel, 1995, 120--139. % (preprint (1994) arXiv:alg-geom/9411018).

\bibitem[LP25a]{LP25a}
S.\,Lee, V.\,Przyjalkowski, \emph{Rationality of smooth Fano threefolds via Mirror Symmetry}, Russ. Math. Surv. \textbf{80}:6 (2025).

\bibitem[LP25b]{LP25b}
S.\,Lee, V.\,Przyjalkowski, \emph{Fibers of Landau--Ginzburg models and rationality}, arXiv:2510.21222.

\bibitem[LP18]{LP16}
V.\,Lunts, V.\,Przyjalkowski, \emph{Landau--Ginzburg Hodge numbers for del Pezzo surfaces},  Adv. Math., \textbf{329} (2018), 189--216.

\bibitem[Prz07]{Prz07}
V.\,Przyjalkowski, \emph{Quantum cohomology of smooth complete intersections in weighted projective spaces and singular toric varieties},
Sb. Math., \textbf{198}:9 (2007), 1325--1340.

\bibitem[Prz13]{Prz13}
V.\,Przyjalkowski. \emph{Weak Landau--Ginzburg models for smooth Fano threefolds}, Izv. Math., \textbf{77}:4 (2013), 135--160.

\bibitem[Prz17]{Prz17}
V.\,Przyjalkowski, \emph{Calabi--Yau compactifications of toric Landau--Ginzburg models for smooth
Fano threefolds}, Sb. Math., \textbf{208}:7 (2017), 992--1013.

\bibitem[Prz18a]{Prz18a}
V.\,Przyjalkowski, \emph{On Calabi--Yau compactifications of toric Landau--Ginzburg models for Fano complete intersections}, Math. Notes,
\textbf{103}:1 (2018), 111--119.

\bibitem[Prz18b]{Prz18b}
V.\,Przyjalkowski, \emph{Toric Landau--Ginzburg models}, Russ. Math. Surv., \textbf{73}:6 (2018), 1033--1118.


\bibitem[PSh15]{PSh15}
V.\,Przyjalkowski, C.\,Shramov, {\it On Hodge numbers of complete intersections and Landau--Ginzburg models},
Int. Math. Res. Not. IMRN, \textbf{2015}:21 (2015), 11302--11332.

\bibitem[PSh20]{PSh20}
V.\,Przyjalkowski, C.\,Shramov, \emph{Hodge level for
weighted complete intersections},
Collect. Math., \textbf{71}:3 (2020), 549--574.

\bibitem[Se08]{Se08}
P.\,Seidel,
\emph{Fukaya categories and Picard-Lefschetz theory},
Zur. Lect. Adv. Math. European Mathematical Society (EMS), Z\"{u}rich, 2008.

\bibitem[SYZ96]{SYZ96}
A.\,Strominger, S.-T.\,Yau, E.\,Zaslow, \emph{Mirror symmetry is T-duality}, Nuclear Physics B, \textbf{479}:1–2 (1996), 243--259.


\end{thebibliography}
\end{document}